\documentclass[11pt,a4paper]{amsart}
\usepackage{amsmath,amssymb,amsthm,color,multicol,setspace}
\usepackage[utf8,utf8x]{inputenc}

\usepackage{fullpage}

\usepackage[pdftitle = {Subpolygons\ in\ Conway-Coxeter\ frieze patterns}]{hyperref}
\usepackage{longtable}
\usepackage{xy,amscd}
\usepackage{comment}
\usepackage{epsfig}
\usepackage{rotating}
\usepackage{xspace}
\usepackage{enumitem}
\usepackage{dashbox}
\usepackage{float}
\xyoption{all}
\usepackage{tikz}
\usetikzlibrary{arrows,decorations.pathmorphing,decorations.pathreplacing,positioning,shapes.geometric,shapes.misc,decorations.markings,decorations.fractals,calc,patterns}

\newtheorem{Lemma}{Lemma}[section]
\newtheorem{Theorem}[Lemma]{Theorem}
\newtheorem{Proposition}[Lemma]{Proposition}

\newtheorem*{Theorem*}{Theorem}
\newtheorem*{Corollary*}{Corollary}

\theoremstyle{definition}
\newtheorem{Definition}[Lemma]{Definition}

\newtheorem{Remark}[Lemma]{Remark}

\newtheorem{Example}[Lemma]{Example}

\numberwithin{equation}{section}

\newcommand{\CC }{\mathbb{C}}

\DeclareMathOperator{\SL}{SL}

\title[Subpolygons in Conway-Coxeter frieze patterns]
{Subpolygons in Conway-Coxeter frieze patterns}

\author{Michael~Cuntz}
\address{Michael Cuntz, Leibniz Universit\"at Hannover,
Institut f\"ur Algebra, Zahlentheorie und Diskre\-te Mathematik,
Fakult\"at f\"ur Mathematik und Physik,
Welfengarten 1,
D-30167 Hannover, Germany}
\email{cuntz@math.uni-hannover.de}
\urladdr{https://www.iazd.uni-hannover.de/cuntz.html}

\author{Thorsten~Holm}
\address{Thorsten Holm, Leibniz Universit\"at Hannover,
Institut f\"ur Algebra, Zahlentheorie und Diskre\-te Mathematik,
Fakult\"at f\"ur Mathematik und Physik,
Welfengarten 1,
D-30167 Hannover, Germany}
\email{holm@math.uni-hannover.de}
\urladdr{http://www2.iazd.uni-hannover.de/\~{ }tholm}

\keywords{Frieze pattern, tame frieze pattern, quiddity cycle, cluster algebra, polygon, triangulation}

\subjclass[2010]{05E15, 05E99, 13F60, 51M20}

\begin{document}

\begin{abstract}
Friezes with coefficients are maps assigning numbers to the edges and diagonals of a regular polygon such that all Ptolemy relations for crossing diagonals are satisfied. Among these, the classic Conway-Coxeter friezes are the ones where all values are natural numbers and all edges have value 1. Every subpolygon of a Conway-Coxeter frieze yields a frieze with coefficients over the natural numbers. In this paper we give a complete arithmetic criterion for which friezes with coefficients appear as subpolygons of Conway-Coxeter friezes. This generalizes a result of our earlier paper with Peter J{\o}rgensen from triangles to subpolygons of arbitrary size.
\end{abstract}

\maketitle

\section{Introduction}

Frieze patterns are infinite configurations of numbers introduced by
Coxeter \cite{Cox71} in the 1970s, the shape of which is reminiscent of 
friezes which appeared in architecture and decorative art for centuries.
The entries in a frieze pattern have to satisfy a specific rule for each
neighbouring $2\times 2$-determinant. This frieze pattern rule is for example implicitly contained in the structure of smooth toric varieties and has been essential in the study of continued fractions more than a century earlier. It also reappeared some 30 years after Coxeter's definition as the exchange condition in Fomin and Zelevinsky's 
cluster algebras, mathematical structures
which became highly influential in many areas of modern 
mathematics. This connection to cluster algebras initiated an intensive 
renewed interest in frieze patterns in recent years, see \cite{MG15}. Whereas classic frieze
patterns are bounded by rows of 1's, to capture cluster algebras with 
coefficients more general boundary rows and a modified rule for 
$2\times 2$-determinants are needed. The resulting frieze patterns with coefficients
have been suggested by Propp \cite{Propp} and recently their fundamental
properties have been studied in \cite{CHJ}. Among other things, it is proven in
\cite{CHJ} that a frieze pattern with coefficents can be viewed as a map on 
edges and diagonals of a regular polygon (with values in a suitable number
system) satisfying the Ptolemy relations for any pair of crossing diagonals;
we then speak of a frieze with coefficients to distinguish these viewpoints.

For classic frieze patterns this viewpoint was well-known, not least for 
classic frieze patterns over positive integers, where a beautiful result of
Conway and Coxeter \cite{CC73} shows that such frieze patterns are in 
bijection with triangulations of regular polygons.
Any subpolygon of a Conway-Coxeter frieze yields a 
frieze with coefficients over the positive integers. The natural question
arises which friezes with coefficients actually appear as subpolygons of 
Conway-Coxeter friezes. A solution would give new insight into the subtle
arithmetic relations of entries in Conway-Coxeter friezes, and hence 
triangulations of polygons. 

It is a special property of a frieze with coefficents to appear in a Conway-Coxeter
frieze. For instance, we observed in \cite{CHJ} that for every triangle in
a Conway-Coxeter frieze the greatest common divisors of any two of the three 
numbers must be equal. This already rules out many friezes with coefficients. 

Still, there are many friezes with coefficients where the condition on the greatest
common divisors holds for all triangles and then it is a priori difficult to
determine whether such a frieze with coefficients appears in a Conway-Coxeter
frieze, or not. As one main result of \cite{CHJ} we have shown that for triangles 
this happens if and only if the three numbers are all odd or do not have the 
same 2-valuation. 

The aim of this paper is to give a complete solution to this problem for polygons of
arbitrary size, that is, we present a  
characterization of those friezes with coefficients which appear as subpolygons 
in Conway-Coxeter friezes. 

\begin{Theorem*}
Let $\mathcal{C}$ be a frieze with coefficients 
on an $n$-gon over positive integers. Then $\mathcal{C}$ appears as a subpolygon
of some Conway-Coxeter frieze if and only if the following conditions are
satisfied:
\begin{enumerate}
\item \label{cond0:gcd}
For any triangle $(a,b,c)$ in $\mathcal{C}$ we have
$\gcd(a,b)=\gcd(b,c)=\gcd(a,c)$.
\item \label{cond0:p+1}
Let $p<n$ be a prime number. Then for each $(p+1)$-subpolygon $\mathcal{D}$
of $\mathcal{C}$ the labels of edges and diagonals in $\mathcal{D}$ are either
all not divisible by $p$ or they do not all have the same $p$-valuation.
\end{enumerate}
\end{Theorem*}

Combining this result with Proposition \ref{prop:pm}, we obtain the following 
consequence (where $k\cdot \mathcal{E}$ denotes the frieze with coefficients
obtained by multiplying the label of each edge and diagonal 
of $\mathcal{E}$ by $k$).

\begin{Corollary*}
Let $\mathcal{C}$ be a frieze with coefficients on an $n$-gon over the positive integers.
Assume that we have $\gcd(a,b)=\gcd(b,c)=\gcd(a,c)$ for any triangle $(a,b,c)$ in $\mathcal{C}$.
Then there exists a Conway-Coxeter frieze $\mathcal{E}$ such that $\mathcal{C}$ is a subpolygon of $k\cdot \mathcal{E}$ for some positive integer $k$.
\end{Corollary*}

The proof of the main result (see Theorem \ref{thm:mgon} below)
does not use the earlier solution for triangles,
that is, we get \cite[Theorem 5.12]{CHJ} as an immediate corollary of the 
above theorem. The two directions of the if and only if statement of the
theorem are proven separately in Section \ref{sec:proof}. The proof of
sufficiency is constructive, that is, we give an explicit algorithm to
compute a Conway-Coxeter frieze containing a given frieze with coefficients
satisfying Conditions (\ref{cond0:gcd}) and (\ref{cond0:p+1}) as a subpolygon. 
Section \ref{sec:example} contains a detailed example. 
In general, these Conway-Coxeter friezes are not unique. However, 
our algorithm yields all possible Conway-Coxeter friezes that contain a 
given frieze with coefficients, because each step in the induction 
allows choices to be made and this can lead to 
several different extensions.

\section{Frieze patterns with coefficients}
\label{sec:definition}

In this section we collect the necessary definitions and fundamental properties
of frieze patterns with coefficients. This concept goes back to an unpublished
manuscript by Propp \cite{Propp}. A general theory of frieze patterns with
coefficients has recently been developed in \cite{CHJ}. 

Although in this paper we are only dealing with frieze patterns over the natural 
numbers, we reproduce the basic definition from \cite{CHJ} in a more
general form allowing arbitrary complex numbers as entries.

\begin{Definition} \label{def:frieze}
Let $R\subseteq\CC$ be a subset of the complex numbers.
Let $n\in \mathbb{Z}_{\ge 0}$.

A \emph{frieze pattern with coefficients of height $n$} over $R$ 
is an infinite array of the form
\[
\begin{array}{ccccccccccc}
 & \ddots & & & &\ddots  & & & \\
0 & c_{i-1,i} & c_{i-1,i+1} & c_{i-1,i+2} & \cdots & \cdots & c_{i-1,n+i} & c_{i-1,n+i+1} & 0 & & \\
& 0 & c_{i,i+1} & c_{i,i+2} & c_{i,i+3} & \cdots & \cdots & c_{i,n+i+1} & c_{i,n+i+2} & 0 & \\
& & 0 & c_{i+1,i+2} & c_{i+1,i+3} & c_{i+1,i+4} & \cdots & \cdots & c_{i+1,n+i+2} & c_{i+1,n+i+3} & 0 \\
 & & & & \ddots  & & & &\ddots  & 
\end{array}
\]
where we also set $c_{i,i}=0=c_{i,n+i+3}$ for all $i\in \mathbb{Z}$,
such that the following holds:
\begin{enumerate}
    \item[{(i)}] $c_{i,j}\in R$ for all $i\in\mathbb{Z}$ and $i< j< n+i+3$.
    \item[{(ii)}] $c_{i,i+1}\neq 0$ for all $i\in \mathbb{Z}$.
    \item[{(iii)}] For every
(complete) adjacent $2\times 2$-submatrix $\begin{pmatrix} c_{i,j} & c_{i,j+1}\\
c_{i+1,j} & c_{i+1,j+1} \end{pmatrix}$ we have
\begin{equation*}\tag{$E_{i,j}$}\label{eq:local}
c_{i,j} c_{i+1,j+1} - c_{i,j+1}c_{i+1,j} = c_{i+1,n+i+3}c_{j,j+1}.
\end{equation*}
\end{enumerate}
\end{Definition}

\begin{Remark}
\begin{enumerate}
    \item Classic frieze patterns, as introduced by Coxeter \cite{Cox71}, 
are those frieze patterns with coefficients with $c_{i,i+1}=1$ for all $i\in \mathbb{Z}$.
A Conway-Coxeter frieze pattern is a frieze pattern with coefficients over $\mathbb{Z}_{> 0}$ with $c_{i,i+1}=1$ for all $i\in \mathbb{Z}$. A classic result of Conway and Coxeter states that these frieze patterns are in bijection with triangulations of regular polygons, see \cite{CC73}.
\item There is a close connection between frieze patterns and Fomin and
Zelevinsky's cluster algebras. Namely, 
starting with a set of indeterminates on a row in the frieze pattern,
the frieze conditions ($E_{i,j}$) produce the cluster variables of the
cluster algebra (of Dynkin type $A$). Whereas the classic Conway-Coxeter frieze
patterns correspond to cluster algebras without coefficients, the more general
frieze patterns with coefficients are linked to cluster algebras with coefficients.
From the cluster algebras perspective this is the main motivation to study 
frieze patterns with coefficients.
\end{enumerate}
\end{Remark}

In general, there are too many frieze patterns
with coefficients to expect a satisfactory theory, even in the case of classic frieze 
patterns,
see \cite{Cuntz-wild} for an illustration of the case of wild $\SL_3$-frieze patterns. Therefore, it is very common in the literature to
restrict to tame frieze patterns. Many interesting frieze patterns 
are tame, e.g.\ all frieze patterns without zero entries, see 
\cite[Proposition 2.4]{CHJ} for a proof of this well-known fact.

\begin{Definition} \label{def:tame}
Let $\mathcal{C}$ be a frieze pattern with coefficients as in Definition \ref{def:frieze}.
Then $\mathcal{C}$ is called {\em tame} if every complete adjacent $3\times 3$-submatrix of 
$\mathcal{C}$ has determinant 0.
\end{Definition}

The entries of a tame frieze pattern with coefficients are closely linked
by many remarkable equations (in addition to the defining equations ($E_{i,j}$)
in Definition \ref{def:frieze}). We restate some results from \cite{CHJ}
which are relevant for the present paper.

First, the entries in a tame frieze patterns are invariant under a glide symmetry.

\begin{Proposition}(\cite[Theorem 2.4]{CHJ}) \label{thm:glide}
Let $R\subseteq \mathbb{C}$ be a subset.
Let $\mathcal{C}=(c_{i,j})$ be a tame frieze pattern with coefficients over $R$ of height
$n$. Then for all entries of $\mathcal{C}$ we have
$c_{i,j} = c_{j,n+i+3}.
$
\end{Proposition}

This implies that
the triangular region shown
in Figure \ref{fig:glide} yields 
a fundamental domain for the action of the glide symmetry.
Note that the indices of the entries
are in bijection with the edges and diagonals of a regular $(n+3)$-gon (viewed as pairs of vertices). This means that we can view every tame 
frieze pattern with coefficients of height $n$ over $R$ as a map on the edges and
diagonals of a regular $(n+3)$-gon with values in $R$.
\begin{figure}
$$
\begin{array}{ccccccccc}
& ~~ & ~\ddots~ & ~ & ~~ & ~~ & ~ & \ddots~ & ~\\ 
& ~0~ & ~c_{1,2}~ & ~c_{1,3}~ & ~\ldots~ & ~\ldots~ & ~\ldots~ & ~c_{1,n+3}~ & ~0~  \\ 
& & ~0~ & ~c_{2,3}~ & ~c_{2,4}~ & ~\ldots~ & ~\ldots~ & ~c_{2,n+3}~ & ~\ddots~  \\
& &  & ~\ddots~ & ~\ddots~ & ~\ddots~ & ~~ & ~\vdots~ & ~ \\
& &  & ~~ & ~\ddots~ & ~\ddots~ & ~\ddots~ & ~\vdots~ & ~\\
& &  & ~~ & ~~ & ~0~ & ~c_{n+1,n+2}~ & ~c_{n+1,n+3}~ & ~\\
& &  & ~~ & ~~ & ~~ & ~0~ & ~c_{n+2,n+3}~ & ~ \ddots \\
& &  & ~~ & ~~ & ~~ & ~~ & ~0~ & ~\ddots \\
& &  & ~~ & ~~ & ~~ & ~~ & ~~ & ~0~ \\
\end{array}
$$
\caption{Fundamental domain for the glide symmetry of a frieze pattern
with coefficients.\label{fig:glide}}
\end{figure}

\bigskip

\noindent
{\bf Convention:} 
We use the notion (tame) {\em frieze pattern with coefficients} for an infinite array 
as in Definition \ref{def:frieze} and the notion (tame) {\em frieze with coefficients} 
for a corresponding map from edges and diagonals of a regular polygon. 

\bigskip

Secondly, 
the entries in a frieze (pattern) with coefficients satisfy Ptolemy relations,
as visualized in Figure \ref{fig:Ptolemy}.

\begin{Definition} \label{def:ptolemy}
Let $\mathcal{C}=(c_{i,j})$ be a tame frieze with coefficients over $R\subseteq \mathbb{C}$
on a regular $m$-gon. We say that $\mathcal{C}$ satisfies the Ptolemy relation for 
the indices $1\le i\le j\le k\le \ell\le m$ if the following equation holds:
\begin{equation*}\tag{$E_{i,j,k,\ell}$}\label{eq:ptolemy}
c_{i,k} c_{j,\ell} = c_{i,\ell} c_{j,k} + c_{i,j} c_{k,\ell}.
\end{equation*}
\end{Definition}

\begin{figure}
  \begin{tikzpicture}[auto]
    \node[name=s, draw, shape=regular polygon, regular polygon sides=500, minimum size=4cm] {};
    \draw[thick] (s.corner 60) to (s.corner 180);
    \draw[thick] (s.corner 180) to (s.corner 300);
    \draw[thick] (s.corner 300) to (s.corner 400);
    \draw[thick] (s.corner 400) to (s.corner 60);
    \draw[thick] (s.corner 60) to (s.corner 300);
    \draw[thick] (s.corner 180) to (s.corner 400);
    
    \draw[shift=(s.corner 60)]  node[above]  {{\small $i$}};
    \draw[shift=(s.corner 180)]  node[left]  {{\small $j$}};
    \draw[shift=(s.corner 300)]  node[below]  {{\small $k$}};
    \draw[shift=(s.corner 400)]  node[right]  {{\small $\ell$}};
   \end{tikzpicture}
   \caption{The Ptolemy relation  (\ref{eq:ptolemy}). 
   \label{fig:Ptolemy}}
\end{figure}

An old result by Coxeter (see \cite[Equation (5.7)]{Cox71})
states that classic friezes satisfy all Ptolemy relations and this can be extended 
to friezes with coefficients.

\begin{Proposition}(\cite[Theorem 2.6]{CHJ})  \label{thm:ptolemy}
Every tame frieze with coefficients
over some subset $R\subseteq \mathbb{C}$
satisfies all Ptolemy relations.
\end{Proposition}

\section{$n$-gons in Conway-Coxeter friezes}
\label{sec:ngons}

From now on we consider frieze patterns with coefficients over positive integers. 

Let us take any classic Conway-Coxeter 
frieze $\mathcal{C}$ on an $n$-gon, that is, a map
from edges and diagonals of a regular
polygon to the positive integers such that all edges of the $n$-gon are mapped to 1.
Restricting this map to any subpolygon of the
$n$-gon yields a frieze with coefficients. In fact,
the restricted
map still satisfies all Ptolemy relations of the subpolygon. See Figure \ref{fig:cut}
for an example. 

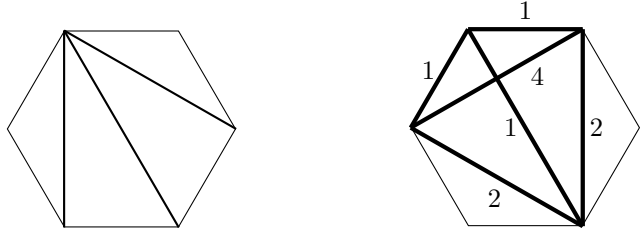
\begin{figure}
  \begin{tikzpicture}[auto]
    \node[name=s, draw, shape=regular polygon, regular polygon sides=6, minimum size=3cm] {};
    \draw[thick] (s.corner 2) to (s.corner 4);
    \draw[thick] (s.corner 2) to (s.corner 5);
    \draw[thick] (s.corner 2) to (s.corner 6);
   \end{tikzpicture}
\mbox{\hskip2cm}
\begin{tikzpicture}[auto]
    \node[name=s, draw, shape=regular polygon, regular polygon sides=6, minimum size=3cm] {};
    \draw[ultra thick] (s.corner 1) to (s.corner 3);
    \draw[ultra thick] (s.corner 3) to (s.corner 5);
    \draw[ultra thick] (s.corner 2) to (s.corner 5);
    \draw[ultra thick] (s.corner 1) to (s.corner 5);
    \draw[ultra thick] (s.corner 1) to (s.corner 2);
    \draw[ultra thick] (s.corner 2) to (s.corner 3);
    
    \draw[shift=(s.side 1)]  node[above]  {{\small 1}};    
    \draw[shift=(s.corner 2)]  node[below left=12pt]  {{\small 1}};
    \draw[shift=(s.corner 4)]  node[above right=5pt]  {{\small 2}};
    \draw[shift=(s.corner 6)]  node[left=10pt]  {{\small 2}};
    \draw[shift=(s.corner 6)]  node[left=42pt]  {{\small 1}};
    \draw[shift=(s.corner 1)]  node[below left=15pt]  {{\small 4}};
   \end{tikzpicture}
   \caption{A frieze with coefficients cut out of a Conway-Coxeter frieze. 
   \label{fig:cut}}
\end{figure}

In \cite{CHJ} we addressed the fundamental question which friezes with coefficients 
actually appear as subpolygons of Conway-Coxeter friezes and obtained the 
following complete answer for the special case of triangles.

\begin{Theorem}(\cite[Theorem 5.12]{CHJ}) \label{thm:triangle}
Let $a,b,c\in \mathbb{N}$. The triple $(a,b,c)$ appears as labels of a triangle
in some Conway-Coxeter frieze if and only if the following two conditions are
satisfied:
\begin{enumerate}
\item $\gcd(a,b)=\gcd(b,c)=\gcd(a,c)$.
\item $\nu_2(a)=\nu_2(b)=\nu_2(c)=0$ or $|\{\nu_2(a),\nu_2(b),\nu_2(c)|>1$ where
$\nu_2(\cdot )$ denotes the 2-valuation.
\end{enumerate}
\end{Theorem}

The main aim of this paper is a generalization of the previous theorem 
to arbitrary subpolygons in Conway-Coxeter friezes. That is, we give  
arithmetic conditions on the entries of a frieze with coefficients
which characterize whether
or not the frieze with coefficients appears as a subpolygon in some 
Conway-Coxeter frieze. 
The following theorem is the main result of this paper.

\begin{Theorem} \label{thm:mgon}
Let $\mathcal{C}$ be a frieze with coefficients 
on an $n$-gon over the positive integers. Then $\mathcal{C}$ appears as a
subpolygon
of some Conway-Coxeter frieze if and only if the following conditions are
satisfied:
\begin{enumerate}
\item \label{cond:gcd}
For any triangle $(a,b,c)$ in $\mathcal{C}$ we have
$\gcd(a,b)=\gcd(b,c)=\gcd(a,c)$.
\item \label{cond:p+1}
Let $p<n$ be a prime number. Then for each $(p+1)$-subpolygon $\mathcal{D}$
of $\mathcal{C}$ the labels of edges and diagonals in $\mathcal{D}$ are either
all not divisible by $p$ or they do not all have the same $p$-valuation.
\end{enumerate}
\end{Theorem}

Note that for the special case $n=3$ this gives precisely the criterion of 
Theorem \ref{thm:triangle}. Actually, our proof of the main result 
Theorem \ref{thm:mgon}
does not need the previous result on triangles from \cite{CHJ}, 
so we get Theorem 
\ref{thm:triangle} as a proper corollary of the new result.

\begin{Example} \label{ex:grobe}
There are friezes with coefficients where each triangle appears 
as a subpolygon of a Conway-Coxeter frieze, but the entire frieze does not.
For instance, consider the square with labels as 
in Figure \ref{fig:exsquare}. 
This gives a frieze with coefficients since the Ptolemy relation is satisfied.
All triangles satisfy the conditions from Theorem \ref{thm:triangle}. However,
for the square itself condition (\ref{cond:p+1}) of Theorem \ref{thm:mgon}
fails for $p=3$, so this square can not appear as a subpolygon of a 
Conway-Coxeter frieze. 

This example was first discovered by Grobe in his Master's thesis 
\cite{Grobe}, by a different argument not using Theorem \ref{thm:mgon}. 
\begin{figure}
\begin{tikzpicture}[auto]
    \node[name=s, draw, shape=regular polygon, regular polygon sides=4, minimum size=3cm] {};
    \draw[thick] (s.corner 1) to (s.corner 3);
    \draw[thick] (s.corner 2) to (s.corner 4);
    
    \draw[shift=(s.side 1)]  node[above]  {{\small 3}};   
    \draw[shift=(s.side 3)]  node[below]  {{\small 3}};
    \draw[shift=(s.side 2)]  node[left]  {{\small 3}};   
    \draw[shift=(s.side 4)]  node[right]  {{\small 3}};
    \draw[shift=(s.side 3)]  node[above right=7pt]  {{\small 3}};
    \draw[shift=(s.side 1)]  node[below right=7pt]  {{\small 6}};
   \end{tikzpicture}
   \caption{A frieze with coefficients which is not a subpolygon of a 
   Conway-Coxeter frieze, but all of whose triangles do appear in
   Conway-Coxeter friezes. 
   \label{fig:exsquare}}
\end{figure}
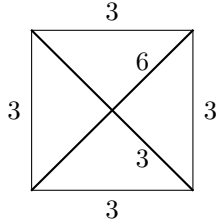
\end{Example}

\begin{Example} \label{ex:p+1}
Condition (\ref{cond:p+1}) imposes to check all prime numbers $p<n$ and the
corresponding $(p+1)$-subpolygons. This is indeed necessary, as the following 
examples show. Note that for $p=3$ this is Example \ref{ex:grobe}
above. 

Let $p$ be any odd prime number. We consider the Conway-Coxeter frieze
on a $(p+1)$-gon given by a fan triangulation,
that is, all diagonals start
at the same vertex; see Figure \ref{fig:fan} for the case $p=11$. 
\begin{figure}
\begin{tikzpicture}[auto]
    \node[name=s, draw, shape=regular polygon, regular polygon sides=12, minimum size=3cm] {};
    \draw[thick] (s.corner 1) to (s.corner 3);
    \draw[thick] (s.corner 1) to (s.corner 4);
    \draw[thick] (s.corner 1) to (s.corner 5);
    \draw[thick] (s.corner 1) to (s.corner 6);
    \draw[thick] (s.corner 1) to (s.corner 7);
    \draw[thick] (s.corner 1) to (s.corner 8);
    \draw[thick] (s.corner 1) to (s.corner 9);
    \draw[thick] (s.corner 1) to (s.corner 10);
    \draw[thick] (s.corner 1) to (s.corner 11);
   \end{tikzpicture}
   \caption{A fan on a dodecagon. 
   \label{fig:fan}}
\end{figure}
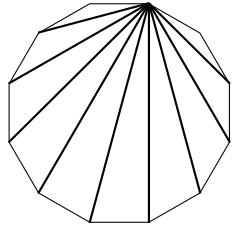
Using Ptolemy relations one checks that the maximal label of a diagonal 
in this frieze is $p-1$ (actually, for each diagonal its label is one more
than the number of diagonals of the fan triangulation it crosses). 
Let $\mathcal{C}$ be the frieze with coefficients
on a $(p+1)$-gon obtained by multiplying the above Conway-Coxeter frieze
by $p$. Then the labels of all edges and diagonals have $p$-valuation 1.
By Theorem \ref{thm:mgon} we see that $\mathcal{C}$ is not a subpolygon 
of a Conway-Coxeter frieze. However, for any prime number $q<p$, all 
$(q+1)$-subpolygons of $\mathcal{C}$ do appear as subpolygons of Conway-Coxeter
friezes, again by Theorem \ref{thm:mgon}; in fact, the corresponding 
$(q+1)$-subpolygons in the Conway-Coxeter frieze clearly satisfy the conditions
of Theorem \ref{thm:mgon} and the validity of these conditions is not 
affected by multiplication with $p$, since $q<p$ are prime numbers. 
\end{Example}

\section{Proof of the main result}
\label{sec:proof}

The aim of this section is to prove Theorem \ref{thm:mgon}.
For clarity, the two directions of the if and only if statement are shown separately.

\subsection{Necessity}
\label{sec:necessary}

We recall from \cite{CHJ} a basic property of Conway-Coxeter friezes, namely that every triangle in a Conway-Coxeter frieze satisfies Condition (\ref{cond:gcd}) of Theorem \ref{thm:mgon}.

\begin{Lemma}(\cite[Lemma 4.3]{CHJ}) \label{lem:gcd}
Let $\mathcal{C}$ be a Conway-Coxeter frieze and $i\le j\le k$. 
Then the greatest common divisor of any two of the numbers $c_{i,j}$,
$c_{j,k}$ and $c_{i,k}$ divides the third number. In particular,
$\gcd(c_{i,j},c_{j,k})=\gcd(c_{i,k},c_{j,k})=\gcd(c_{i,j},c_{i,k})$.
\end{Lemma}

The next step in the proof is to notice that the condition on the $\gcd$'s 
from Lemma \ref{lem:gcd} has implications for the situation where 
$(p+1)$-subpolygons with the same $p$-valuations exist.

\begin{Proposition}\label{prop:pm}
Let $\mathcal{C}$ be a frieze with coefficients on an $n$-gon over the positive integers.
Assume that we have
\begin{equation}\label{cgcd}
\gcd(a,b)=\gcd(b,c)=\gcd(a,c)
\end{equation}
for any triangle $(a,b,c)$ in $\mathcal{C}$
and that $\mathcal{C}$ contains a $(p+1)$-subpolygon $\mathcal{D}$ for a prime number $p$
such that the labels of all edges and diagonals of $\mathcal{D}$ have the 
same $p$-valuation $m$.
Then the label of every edge and diagonal of $\mathcal{C}$ is divisible by $p^m$.
\end{Proposition}

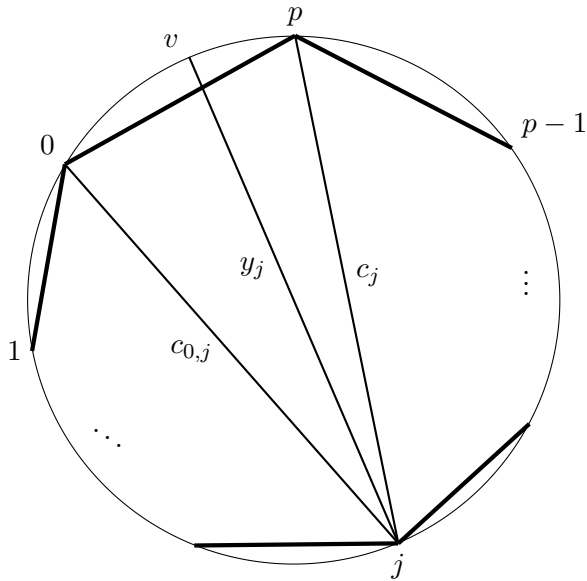
\begin{figure} 
\begin{tikzpicture}[auto]
    \node[name=s, draw, shape=regular polygon, regular polygon sides=600, minimum size=7cm] {};
    \draw[ultra thick] (s.corner 1) to (s.corner 100);
    \draw[ultra thick] (s.corner 100) to (s.corner 170);
    \draw[ultra thick] (s.corner 265) to (s.corner 340);
    \draw[ultra thick] (s.corner 340) to (s.corner 405);
    \draw[ultra thick] (s.corner 1) to (s.corner 510);
    
    \draw[thick] (s.corner 1) to (s.corner 340);
    \draw[thick] (s.corner 40) to (s.corner 340);
    \draw[thick] (s.corner 100) to (s.corner 340);

    \draw[shift=(s.corner 340)]  node[below]  {{$j$}};   
    \draw[shift=(s.corner 1)]  node[above]  {{$p$}};
    \draw[shift=(s.corner 100)]  node[above left]  {{$0$}};   
    \draw[shift=(s.side 200)]  node[right=7pt]  {{$\ddots$}};
    \draw[shift=(s.side 460)]  node[left=7pt]  {{$\vdots$}};
    \draw[shift=(s.corner 40)]  node[above left]  {{$v$}};
    \draw[shift=(s.corner 510)]  node[above right]  {{$p-1$}};
    \draw[shift=(s.corner 170)]  node[left]  {{$1$}};   
    \draw[shift=(s.side 460)]  node[left=63pt]  {{$c_j$}};
    \draw[shift=(s.corner 170)]  node[right=48pt]  {{$c_{0,j}$}};  
    \draw[shift=(s.corner 140)]  node[right=75pt]  {{$y_j$}};   
   \end{tikzpicture}
   \caption{A $(p+1)$-subpolygon in a larger frieze. 
   \label{fig:p+1}}
\end{figure}

\begin{proof}
Let $\mathcal{C}$ and $\mathcal{D}$ be as above; we denote the vertices of $\mathcal{D}$ by $0,\ldots,p$.
We proceed by induction on $m$. If $m=0$, then the claim is trivial, so consider $m>0$.

Assume first that every diagonal $(i,v)$ for $i=0,\ldots,p$ and $v$ not a vertex of $\mathcal{D}$ is divisible by $p$.
Then if $v,w$ are vertices of $\mathcal{C}$ not in $\mathcal{D}$, then the 
label of the diagonal $(v,w)$ is divisible by $p$ as well by assumption (\ref{cgcd}) since $(c_{0,v},c_{v,w},c_{0,w})$ is a triangle.
Dividing the labels of all edges and diagonals of $\mathcal{C}$ by $p$ we obtain a frieze with coefficients $\mathcal{C'}$ satisfying the assumption of the proposition with $m-1$ instead of $m$, thus we are finished by induction.

We may thus now assume without loss of generality that there exists a vertex $v$ such that the diagonal $(v,p)$ is not divisible by $p$, see Figure \ref{fig:p+1}.
For $j=0,1,\ldots,p$ we set $c_j:=c_{p,j}$ and $y_j:=c_{v,j}$ for abbreviation.

For $j=1,\ldots,p-1$ 
the Ptolemy relation for the crossing diagonals $(0,p)$ and $(v,j)$ of 
$\mathcal{C}$ gives
$$c_0y_j = y_0c_j + c_{0,j}.
$$
Dividing this equation by $p^m$ leads to
\begin{equation} \label{eq1}
    c'_0y_j = y_0c'_j + c'_{0,j}
\end{equation}
where $c'_0=\frac{c_0}{p^m}$, $c'_j=\frac{c_j}{p^m}$ and
$c'_{0,j}=\frac{c_{0,j}}{p^m}$. By assumption on $\mathcal{D}$, none of 
these three positive integers is divisible by $p$. 
In addition, note that $y_j$ is not divisible by $p$ by assumption (\ref{cgcd}), since
$(y_p,y_j,c_j)$ are the labels of a triangle in $\mathcal{C}$ and 
$y_p=c_{v,p}$ is not divisible by $p$.
Then Equation (\ref{eq1}) implies
\begin{equation} \label{eq:residue}
y_0\not\equiv -(c'_j)^{-1} c'_{0,j}\,\,(\operatorname{mod}\, p)
\mbox{\hskip1cm 
for all $j=1,\ldots,p-1$}.
\end{equation}
On the other hand, for any $i< j$, dividing the Ptolemy relation for the crossing diagonals
$(0,j)$ and $(p,i)$ by $p^{2m}$ yields 
$$c'_{0,j} c'_i - c'_{0,i}c'_j = c'_0c'_{i,j} \not\equiv 0
\,\,(\operatorname{mod}\, p).
$$
That is, the residue classes modulo $p$
appearing on the right of (\ref{eq:residue})
are pairwise different for $j=1,\ldots,p-1$. Hence the conditions in
(\ref{eq:residue}) rule out all nonzero residue classes modulo $p$
for $y_0$, but
since $y_0$ is not divisible by $p$, this leaves no choice for $y_0$. 
This is a contradiction and thus this case cannot occur.
\end{proof}

We now show that Conditions (\ref{cond:gcd}) and (\ref{cond:p+1}) are necessary for a frieze with
coefficients to appear as a subpolygon of a Conway-Coxeter frieze.
So assume that $\mathcal{C}$ is a frieze with coefficients that appears as a subpolygon of a Conway-Coxeter frieze $\mathcal{E}$.

By Lemma \ref{lem:gcd}, Condition (\ref{cond:gcd}) is satisfied in $\mathcal{E}$, thus satisfied in $\mathcal{C}$ as well.

Now assume that $\mathcal{C}$ contains a $(p+1)$-subpolygon $\mathcal{D}$ for a prime number $p$
such that the labels of all edges and diagonals of $\mathcal{D}$ have the same $p$-valuation $m$.
Proposition \ref{prop:pm} tells us that then the labels of all edges and
diagonals of $\mathcal{E}$ are divisible by $p^m$. Since the edges of the
Conway-Coxeter frieze $\mathcal{E}$ are labelled by $1$, we obtain $m=0$,
that is, the labels of all edges and diagonals of $\mathcal{D}$ are
not divisible by $p$, and condition (\ref{cond:p+1}) holds.

\subsection{Sufficiency}

It remains to prove the sufficiency statement of Theorem \ref{thm:mgon}. Let
$\mathcal{C}$ be a frieze with coefficients over $\mathbb{Z}_{>0}$ on an 
$n$-gon satisfying conditions (\ref{cond:gcd}) and (\ref{cond:p+1}). We have to show
that $\mathcal{C}$ can be extended to a Conway-Coxeter frieze. 

If all boundary edges have label 1 then $\mathcal{C}$ is itself a Conway-Coxeter
frieze and we are done. So assume that $\mathcal{C}$ has a boundary edge with
label $c_0>1$. 
The idea of the proof is to proceed inductively. That is, we aim to construct
a frieze with coefficients $\widetilde{\mathcal{C}}$ over $\mathbb{Z}_{>0}$ 
on an $(n+1)$-gon with the following properties:
\begin{enumerate}
    \item[{(i)}]  $\widetilde{\mathcal{C}}$
    contains $\mathcal{C}$ as a subpolygon. 
    \item[{(ii)}] The edges attached to the new vertex have labels $1$ and $y_0$ where 
    $0< y_0 < c_0$.
    \item[{(iii)}] $\widetilde{\mathcal{C}}$ still satisfies Conditions 
    (\ref{cond:gcd}) and (\ref{cond:p+1}).
\end{enumerate} 
Carrying out this procedure inductively for each boundary edge of $\mathcal{C}$
eventually produces 
a frieze with coefficients with all boundary edges having label 1, that is, a 
Conway-Coxeter frieze containing $\mathcal{C}$ as a subpolygon. We will give an
explicit algorithm to determine such a frieze with coefficients $\widetilde{\mathcal{C}}$,
that is, the proof of this direction is constructive. 
\smallskip

We label the vertices of the $n$-gon by $0,1,\ldots,n-1$ in counterclockwise order,
such that the edge with label $c_0$ has vertices $0$ and $n-1$, see Figure \ref{fig:ext}. 
\begin{figure} 
\begin{tikzpicture}[auto]
    \node[name=s, draw, shape=regular polygon, regular polygon sides=11, minimum size=7cm] {};
    \draw[dashed] (s.corner 1) to (s.corner 3);
    \draw[dashed] (s.corner 1) to (s.corner 5);
    \draw[dashed] (s.corner 1) to (s.corner 7);
    \draw[ultra thick] (s.corner 2) to (s.corner 11);
    \draw[ultra thick] (s.corner 3) to (s.corner 11);
    \draw[ultra thick] (s.corner 5) to (s.corner 11);
    \draw[ultra thick] (s.corner 7) to (s.corner 11);
    
    \draw[shift=(s.side 3)]  node[right=40pt]  {{$y_j$}};
    \draw[shift=(s.side 3)]  node[right=70pt]  {{$c_j$}};
    \draw[shift=(s.side 9)]  node[left=30pt]  {{$c_{i_p}$}};
    \draw[shift=(s.side 11)]  node[above right]  {{$1$}};
    \draw[shift=(s.side 1)]  node[above left]  {{$y_0$}};
    \draw[shift=(s.corner 9)]  node[left=79pt]  {{$y_{i_p}$}};
    \draw[shift=(s.corner 2)]  node[left]  {{$0$}};  
    \draw[shift=(s.corner 3)]  node[left]  {{$1$}};
    \draw[shift=(s.corner 11)]  node[right]  {{$n-1$}};
    \draw[shift=(s.corner 1)]  node[above]  {{$n$}};
    \draw[shift=(s.corner 4)]  node[above right=20pt]  {{$\ddots$}};
    \draw[shift=(s.side 5)]  node[above right=20pt]  {{$\ddots$}};
    \draw[shift=(s.corner 5)]  node[left]  {{$j$}};   
    \draw[shift=(s.corner 7)]  node[below right]  {{$i_p$}};
   \end{tikzpicture}
   \caption{Extending a frieze with coefficients. 
   \label{fig:ext}}
\end{figure}
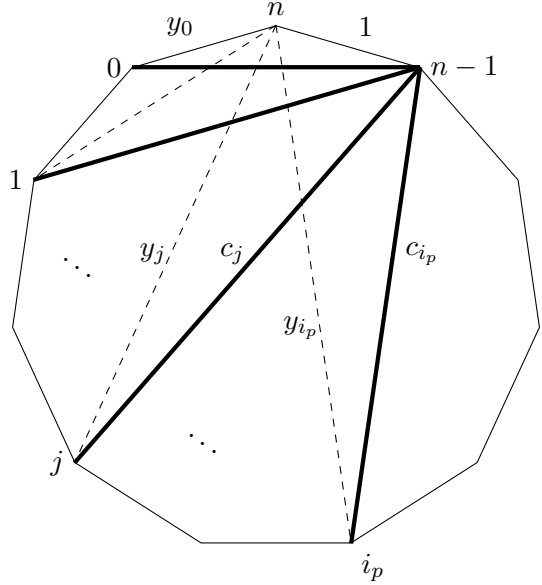
We set $c_j:= c_{j,n-1}$ for $0\le j\le n-2$, see the ultra thick lines in
Figure \ref{fig:ext}. We aim to find suitable labels $y_j:=c_{j,n}$ for the 
new edges and diagonals in the larger frieze with coefficients $\widetilde{\mathcal{C}}$
(the dashed lines in Figure \ref{fig:ext})
such that all Ptolemy relations in $\widetilde{\mathcal{C}}$ are satisfied. 

For computing suitable positive integers $y_j$, 
we consider each prime power divisor of $c_0$ separately and eventually use the 
Chinese Remainder Theorem.

Let $p$ be a prime divisor of $c_0$ and $\ell:=\nu_p(c_0)$ be the $p$-valuation
(that is, $p^{\ell}$ divides $c_0$ but $p^{\ell+1}$ does not divide $c_0$). 
We set
$$m:= \min \{\nu_p(c_i)\mid 0\le i\le n-2\},
$$
and we choose a vertex $i_p$ with $\nu_p(c_{i_p})=m$. Note that for every vertex $j$
in $\mathcal{C}$ we have $p^m\mid c_j$ (by minimality of $m$) and 
also $p^m\mid c_{i_p,j}$ (by Condition (\ref{cond:gcd}) for $\mathcal{C}$). 

For any positive integer $u$ we define $u'$ by $u=p^{\nu_p(u)} u'$.
\medskip

We first want to determine a suitable label $y_{i_p}$.

\begin{Lemma} \label{lem:yip}
With the above notation there are positive integers $y_{i_p}$ such that the following 
conditions are satisfied.
\begin{enumerate}
\item[{(i)}] $y_{i_p} \not \equiv 0\,\,(\operatorname{mod}\,p)$.
\item[{(ii)}] For every vertex $j$ such that $p\nmid \frac{c_{i_p,j}}{p^m}$ and $p\nmid \frac{c_{j}}{p^m}$ we have
\begin{equation} \label{eq:ruleout}
    \left\{ 
\begin{array}{lll} 
c'_j y_{i_p} - c'_{i_p,j}\not \equiv 0\,\,(\operatorname{mod}\,p) & & \mbox{~~if $j<i_p$} \\
c'_j y_{i_p} + c'_{i_p,j}\not \equiv 0\,\,(\operatorname{mod}\,p) & & \mbox{~~if $i_p<j$}
\end{array} \right. .
\end{equation}
\end{enumerate}
\end{Lemma}

\begin{proof}
We consider the nonzero residue classes modulo $p$ and show that for the elements in at least one
residue class the conditions of the lemma hold. Let $j$ be a vertex such that
$p\nmid \frac{c_{i_p,j}}{p^m}$ and $p\nmid \frac{c_{j}}{p^m}$. Then the second condition 
in the lemma rules out the residue class 
$\pm (c'_j)^{-1} c'_{i_p,j}\,\,(\operatorname{mod}\,p)$ to be chosen for $y_{i_p}$.
\smallskip

{\em Claim:} Let vertices $i$ and $j$ both satisfy the assumptions in the second
condition of the lemma. Then 
(\ref{eq:ruleout}) rules out the same residue class modulo $p$
if
$p\mid \frac{c_{i,j}}{p^m}$ and different residue classes modulo $p$ otherwise. 
\smallskip

{\em Proof of the claim:} We can assume $i<j$. There are different cases according
to the location of the vertex $i_p$. We give the details for the case $i<j<i_p$,
the other cases $i<i_p<j$ and $i_p<i<j$ are completely analogous. 

Dividing the Ptolemy relation for the crossing diagonals $(i,i_p)$ and $(j,n-1)$
by $p^{2m}$ yields 
$$c'_j c'_{i,i_p} = c'_i c'_{j,i_p} + c'_{i_p} \frac{c_{i,j}}{p^m}.
$$
This implies 
$$(c'_i)^{-1}c'_{i,i_p} - (c'_j)^{-1}c'_{j,i_p} \equiv 
(c'_i)^{-1}(c'_j)^{-1}c'_{i_p} \frac{c_{i,j}}{p^m}\,\,\,
(\operatorname{mod} p)
$$
which is congruent to 0 if and only if $p$ divides $\frac{c_{i,j}}{p^m}$. Since 
the summands on the left hand side are the values ruled out for $y_{i_p}$
by (\ref{eq:ruleout}), the claim follows.
\medskip

By assumption, the frieze with coefficients $\mathcal{C}$ satisfies Condition 
(\ref{cond:p+1}). This means that there cannot be $p-1$ different vertices 
$j_1,\ldots, j_{p-1}$ satisfying the assumptions in the second condition and such that
$p\nmid \frac{c_{j_r,j_s}}{p^m}$. This implies that by condition (\ref{eq:ruleout})
not all residue classes are ruled out and hence we can choose positive integers $y_{i_p}$
as claimed. 
\end{proof}

Using a suitable value for $y_{i_p}$ as in Lemma \ref{lem:yip} we now want to
look for suitable values for the other new diagonals $y_j$, such that 
the Ptolemy relations in the larger polygon $\widetilde{\mathcal{C}}$ can 
be satisfied.

\begin{Lemma} \label{lem:yj}
We keep the above notation and fix a positive integer $y_{i_p}$ as in Lemma \ref{lem:yip}. 
Then for every integer $y_j$ in the residue class 
$$
\left\{ \begin{array}{lll}
(c'_{i_p})^{-1} \left( \frac{c_j}{p^m} y_{i_p} - \frac{c_{i_p,j}}{p^m}\right)\quad
(\operatorname{mod}\,p^{\ell}) && \text{ if } j<i_p \\
(c'_{i_p})^{-1} \left( \frac{c_j}{p^m} y_{i_p} + \frac{c_{i_p,j}}{p^m}\right) \quad
(\operatorname{mod}\,p^{\ell})
&& \text{ if } i_p<j
\end{array} \right.
$$
the following holds.
\begin{enumerate}
\item[{(a)}] $p\nmid \gcd(y_j,c_j)$.
\item[{(b)}] For every $0\le i<j$ we have
$c_i y_j \equiv c_j y_i + c_{i,j}\,\,(\operatorname{mod}\,p^{\ell})$.
\end{enumerate}
\end{Lemma}

\begin{proof}
We have a congruence 
\begin{equation} \label{eq:modpl}
    c'_{i_p} y_j \equiv \frac{c_j}{p^m} y_{i_p} \mp \frac{c_{i_p,j}}{p^m}\,\,
(\operatorname{mod}\,p^{\ell}).
\end{equation}
\medskip

(a) We consider various cases. 

Case 1: Suppose $p\mid \frac{c_{i_p,j}}{p^m}$ .
Then condition (\ref{cond:gcd}) for $\mathcal{C}$ implies that 
$p\nmid \frac{c_{j}}{p^m}$. Moreover, $p\nmid y_{i_p}$ by Lemma \ref{lem:yip}.
Then (\ref{eq:modpl}) gives $p\nmid y_j$. In particular, $p\nmid \gcd(y_j,c_j)$.
\smallskip

Case 2: Suppose $p\mid \frac{c_{j}}{p^m}$. 
Then condition (\ref{cond:gcd}) for $\mathcal{C}$ implies that 
$p\nmid \frac{c_{i_p,j}}{p^m}$. Then 
(\ref{eq:modpl}) yields
that $p\nmid y_j$. In particular, $p\nmid \gcd(y_j,c_j)$. 
\smallskip

Case 3: Suppose $p\nmid \frac{c_{i_p,j}}{p^m}$ and $p\nmid \frac{c_{j}}{p^m}$.
According to the choice of $y_{i_p}$ in 
Lemma \ref{lem:yip}, the right hand side of (\ref{eq:modpl}) is 
invertible modulo $p^{\ell}$. Hence the left hand side is invertible as well.
Thus, $p\nmid y_j$ and in particular $p\nmid \gcd(y_j,c_j)$. 
\medskip

(b) Due to the signs in the definition of $y_j$, 
there are separate cases. We present the argument for the case $i<i_p<j$, the other cases
$i<j<i_p$ and $i_p<i<j$ are very similar. 

By definition of $y_j$ we have
\begin{eqnarray*}
c_i y_j - c_j y_i & \equiv & c_i (c'_{i_p})^{-1} \left(\frac{c_j}{p^m} y_{i_p} + \frac{c_{i_p,j}}{p^m}\right)
- c_j (c'_{i_p})^{-1}\left( \frac{c_i}{p^m} y_{i_p} - \frac{c_{i_p,i}}{p^m}\right)\,\,\,
(\operatorname{mod}\,p^{\ell}) \\
 & \equiv & (c'_{i_p})^{-1} \left( c_i \frac{c_{i_p,j}}{p^m} + c_j \frac{c_{i_p,i}}{p^m}\right)
 \,\,\,(\operatorname{mod}\,p^{\ell}) .
\end{eqnarray*}
The Ptolemy relation in $\mathcal{C}$ for the crossing diagonals $(i,j)$ and $(i_p,n-1)$
reads 
$$c_{i,j} c_{i_p} = c_i c_{i_p,j} + c_{i_p,i}c_j.
$$
Dividing this equation by $p^m$ and plugging it into the above congruence gives
$$
c_i y_j - c_j y_i  \equiv (c'_{i_p})^{-1} c_{i,j} \frac{c_{i_p}}{p^m}
\equiv (c'_{i_p})^{-1} c_{i,j} c'_{i_p}
\equiv c_{i,j}\,\,\,(\operatorname{mod}\,p^{\ell}),
$$
as claimed.
\end{proof}

We have now constructed 
residue classes for $y_0,y_1,\ldots, y_{n-2}$ modulo $p^{\nu_p(c_0)}$
for each prime divisor $p$ of $c_0$, satisfying the conditions in Lemma 
\ref{lem:yip} and Lemma \ref{lem:yj}. 
Then the Chinese Remainder Theorem 
yields residue classes for $y_0,y_1,\ldots, y_{n-2}$ modulo $c_0$,
which according to Lemma \ref{lem:yj}\,(b) in particular satisfy
$$c_0y_j \equiv c_jy_0 + c_{0,j}\,\,(\operatorname{mod}\,c_0)
$$
for all $j=1,\ldots,n-2$. For $y_0$ we choose 
the smallest positive representative in this residue class, that is, we
have $0< y_0 < c_0$. (In fact, by Lemma \ref{lem:yj}\,(a) we have that
$y_0$ and $c_0$ are coprime, in particular, $y_0$ is nonzero.) 
Recall that this is needed to make the inductive strategy work. 
In particular, for this choice  we have that for each vertex $j=1,\ldots, n-2$ the 
number 
\begin{equation} \label{eq:yj}
    y_j = \frac{c_jy_0 + c_{0,j}}{c_0}
    \end{equation}
is a positive integer. 

Finally, to make the inductive strategy work, we have to show that 
$\widetilde{\mathcal{C}}$ is indeed a frieze with coefficients and that 
$\widetilde{\mathcal{C}}$ satisfies conditions 
(\ref{cond:gcd}) and (\ref{cond:p+1}) of Theorem \ref{thm:mgon}.

\begin{Proposition}
With the above notations and definitions, the following holds.
\begin{enumerate}
    \item[{(a)}] All Ptolemy relations in 
$\widetilde{\mathcal{C}}$ are satisfied, that is, $\widetilde{\mathcal{C}}$
is a frieze with coefficients over $\mathbb{Z}_{>0}$.
\item[{(b)}] $\widetilde{\mathcal{C}}$ satisfies conditions 
(\ref{cond:gcd}) and (\ref{cond:p+1}) of Theorem \ref{thm:mgon}. 
\end{enumerate}
\end{Proposition}

\begin{proof}
(a) The Ptolemy relations not involving any of the new diagonals with
label $y_j$ are Ptolemy relations of $\mathcal{C}$ and hold by assumption 
since $\mathcal{C}$ is a frieze with coefficients. 

For crossings of diagonals labelled $y_j$ with the diagonal with label $c_0$
the Ptolemy relation holds by definition of $y_j$ in (\ref{eq:yj}). 

Let $(i,k)$ be a diagonal in $\mathcal{C}$ crossing the new diagonal with label 
$y_j$. Using the formula in (\ref{eq:yj}) and Ptolemy relations in $\mathcal{C}$
we get
\begin{eqnarray*}
y_kc_{i,j} + y_i c_{j,k} & = & \frac{c_ky_0+c_{0,k}}{c_0} c_{i,j}
+ \frac{c_iy_0+c_{0,i}}{c_0} c_{j,k} \\
& = & \frac{1}{c_0} ( y_0 (c_kc_{i,j} + c_ic_{j,k}) + 
c_{0,k} c_{i,j} + c_{0,i} c_{j,k}) \\
& = & \frac{1}{c_0} ( y_0 c_jc_{i,k} + c_{0,j}c_{i,k}) 
 =  \frac{c_j y_0+c_{0,j}}{c_0} c_{i,k}  =  y_j c_{i,k}. 
\end{eqnarray*}
Note that in particular we also obtain $c_i y_j = c_j y_i + c_{i,j}$ for all $i,j$.

(b) For condition (\ref{cond:gcd}) we have to consider the triangles in
$\widetilde{\mathcal{C}}$ which are not already in $\mathcal{C}$. 
There are different types of triangles to consider. 

The triangle $(1,y_0,c_0)$ satisfies condition (\ref{cond:gcd}) by Lemma 
\ref{lem:yj}\,(a). For a triangle $(1,y_j,c_j)$ with $j\neq 0$ we know again
by Lemma \ref{lem:yj}\,(a) that $p\nmid \gcd(y_j,c_j)$ for all prime divisors
$p$ of $c_0$. Suppose $q$ is a prime number dividing $y_j$ and $c_j$ but
$q\nmid c_0$. Then $q\mid c_{0,j}$ by (\ref{eq:yj}). Thus $q$ is a common divisor
of $c_{0,j}$ and $c_j$. But the triangle $(c_0,c_{0,j},c_j)$ in $\mathcal{C}$
satisfies Condition (\ref{cond:gcd}), so $q\mid c_0$, a contradiction. Thus 
we have shown that $\gcd(y_j,c_j)=1$ and the triangle $(1,y_j,c_j)$ satisfies
Condition (\ref{cond:gcd}).

The other new triangles in $\widetilde{\mathcal{C}}$ are of the form
$(y_i,c_{i,j},y_j)$. We use the Ptolemy relation $c_i y_j = c_j y_i + c_{i,j}$.
Let $d$ be a common divisor of $y_i$ and $c_{i,j}$. Then $d$ divides $c_iy_j$.
But $y_i$ and $c_i$ are coprime as shown in the previous paragraph, so $d$ divides $y_j$, as desired. Similarly, if $d$ is a common divisor of $c_{i,j}$ and $y_j$, then $d$ divides $y_i$. Finally, if $d$ is a common divisor 
of $y_i$ and $y_j$ then $d$ divides $c_{i,j}$. 

So Condition (\ref{cond:gcd}) holds for all triangles in 
$\widetilde{\mathcal{C}}$.
\smallskip

For Condition (\ref{cond:p+1}) we have to consider all possible $(q+1)$-gons
in $\widetilde{\mathcal{C}}$ for all prime numbers $q<n+1$. The subpolygons 
in $\mathcal{C}$ satisfy Condition (\ref{cond:p+1}) by assumption. 
So it suffices to consider $(q+1)$-gons $\mathcal{D}$ 
involving the new vertex $n$ and $q$
vertices of $\mathcal{C}$. Suppose that all edges and diagonals in $\mathcal{D}$ 
have the same positive $q$-valuation. Note that in $\widetilde{\mathcal{C}}$ there is a boundary
edge with label 1 attached to $\mathcal{D}$. But  
we have shown in the proof of necessity in
Subsection \ref{sec:necessary} that such a configuration leads to a contradiction. 
Therefore, $\widetilde{\mathcal{C}}$ satisfies Condition (\ref{cond:p+1}),
as needed for the inductive procedure to work. 

This completes the proof of the sufficiency direction in Theorem \ref{thm:mgon}.
\end{proof}

\section{A worked example}
\label{sec:example}

The proof of our main Theorem \ref{thm:mgon} is constructive.
In this section we go through an explicit example
to illustrate how the methods in the proof of the previous section yield
an algorithm to determine a Conway-Coxeter frieze having a given
frieze with coefficients as a subpolygon. 

Let $\mathcal{C}$ be the frieze with coefficients given in Figure \ref{fig:fwc}.
\begin{figure}
\begin{tikzpicture}[auto]
    \node[name=s, draw, shape=regular polygon, regular polygon sides=4, minimum size=3.7cm] {};
    \draw[thick] (s.corner 1) to (s.corner 3);
    \draw[thick] (s.corner 2) to (s.corner 4);
    
    \draw[shift=(s.side 1)]  node[above]  {{$12$}};
    \draw[shift=(s.side 3)]  node[below]  {{$4$}};
    \draw[shift=(s.side 2)]  node[left]  {{$2$}};
    \draw[shift=(s.side 4)]  node[right]  {{$2$}};
    \draw[shift=(s.side 4)]  node[above left=22pt]  {{$2$}};
    \draw[shift=(s.side 3)]  node[above right=9pt]  {{$26$}};
   \end{tikzpicture}
   \caption{A frieze with coefficients on a square. 
   \label{fig:fwc}}
\end{figure}
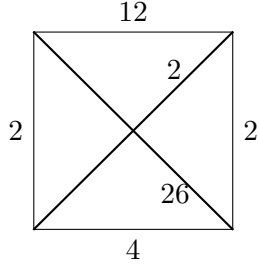
One checks that $\mathcal{C}$ satisfies conditions (\ref{cond:gcd}) and (\ref{cond:p+1})
of Theorem \ref{thm:mgon}, therefore $\mathcal{C}$ can be realized as a subpolygon of
some Conway-Coxeter frieze. We illustrate here how to determine such a Conway-Coxeter
frieze using the methods from the previous section.
\smallskip

Each boundary edge of $\mathcal{C}$ has to be extended. We start with the boundary edge
with label $12$. With the notation as in the previous section we set $c_0=12$, and hence
$c_1=2$ and $c_2=2$. We consider each prime divisor of $c_0$ separately. 

For $p=2$ we have $m=\min\{\nu_2(c_i)\mid 0\le i\le 2\}=1$, and we choose the vertex
$i_2=2$.
We want to determine a suitable value for $y_{i_2}$, using Lemma \ref{lem:yip}.
One checks that no restriction occurs here, so we can choose 
$y_{i_2}\equiv 1\,(\operatorname{mod}\,4)$.

For $p=3$ we have $m=\min\{\nu_3(c_i)\mid 0\le i\le 2\}=0$, and we choose $i_3=2$.
One checks that Lemma \ref{lem:yip} only imposes a restriction for $j=1$, namely 
$y_{i_3}\not\equiv 2\,(\operatorname{mod}\,3)$. So we choose 
$y_{i_3}\equiv 1\,(\operatorname{mod}\,3)$.

The next step now is to compute a suitable value for 
$y_{0}\,(\operatorname{mod}\,c_0)$ by using
Lemma \ref{lem:yj}.

For $p=2$, we have
$$y_0 \equiv (c'_{i_2})^{-1} \left( \frac{c_0}{2} y_{i_2} - \frac{c_{2,0}}{2}\right)
\equiv 1\cdot (6\cdot 1 - 13) \equiv 1 \,(\operatorname{mod}\,4).
$$

Similarly, for $p=3$ one gets
$$y_0 \equiv (c'_{i_3})^{-1} \left( \frac{c_0}{1} y_{i_3} - \frac{c_{2,0}}{1}\right)
\equiv 2\cdot (12\cdot 1 - 26) \equiv 2 \,(\operatorname{mod}\,3).
$$

By the Chinese Remainder Theorem we obtain
$$y_{0}\equiv 5 \,(\operatorname{mod}\,12).
$$

Now we can use Equation (\ref{eq:yj}) to compute the values for $y_1$ and $y_2$, namely
$$y_1=\frac{c_1y_0+c_{0,1}}{c_0} = \frac{2\cdot 5 +2}{12} = 1\hskip0.5cm
\mbox{and}\hskip0.5cm
y_2=\frac{c_2y_0+c_{0,2}}{c_0} = \frac{2\cdot 5 +26}{12} = 3.
$$
Thus we obtain the frieze with coefficients as in Figure \ref{fig:fwcext1}, where we
draw thick lines for diagonals with label 1, that is, for those diagonals which will
appear in the final triangulation.

\medskip

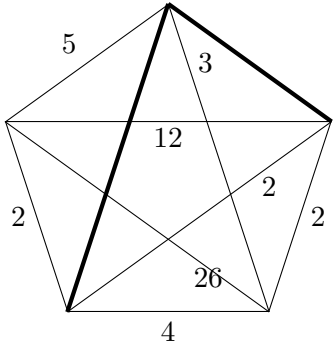
\begin{figure} 
\begin{tikzpicture}[auto]
    \node[name=s, draw, shape=regular polygon, regular polygon sides=5, minimum size=4.5cm] {};
    \draw[ultra thick] (s.corner 1) to (s.corner 3);
    \draw[ultra thick] (s.corner 1) to (s.corner 5);
    \draw[] (s.corner 1) to (s.corner 4);
    \draw[] (s.corner 2) to (s.corner 4);
    \draw[] (s.corner 2) to (s.corner 5);
    \draw[] (s.corner 3) to (s.corner 5);
    
    \draw[shift=(s.side 1)]  node[above left]  {{$5$}};
    \draw[shift=(s.side 3)]  node[below]  {{$4$}};
    \draw[shift=(s.side 2)]  node[left]  {{$2$}};
    \draw[shift=(s.side 4)]  node[right]  {{$2$}};
    \draw[shift=(s.corner 1)]  node[below=43pt]  {{$12$}};
    \draw[shift=(s.side 3)]  node[above right=8pt]  {{$26$}};
    \draw[shift=(s.side 5)]  node[left=10pt]  {{$3$}};
    \draw[shift=(s.corner 4)]  node[above=40pt]  {{$2$}};
   \end{tikzpicture}
   \caption{First step of the extension of $\mathcal{C}$.
   \label{fig:fwcext1}}
\end{figure}

Now we extend further at the boundary edge with label 5. 
We then have $c_0=5$, $c_1=1$, $c_2=3$ and $c_4=1$. For the relevant
prime number $p=5=c_0$, we have 
$m=\min\{\nu_5(c_i)\mid 0\le i\le 3\}=0$ and we choose the vertex $i_5=3$. 

The conditions in Lemma \ref{lem:yip} give restrictions for $j=1$ and $j=2$, 
namely $y_{i_5}\not\equiv 2\,(\operatorname{mod}\,5)$ and
$y_{i_5}\not\equiv 4\,(\operatorname{mod}\,5)$. So we can choose 
$y_{i_5}\equiv 1\,(\operatorname{mod}\,5)$.

With Lemma \ref{lem:yip} we then compute the value for the new edge as
$$y_0 \equiv (c'_{i_5})^{-1} \left( \frac{c_0}{1} y_{i_5} - \frac{c_{3,0}}{1}\right)
\equiv 5-12 \equiv 3 \,(\operatorname{mod}\,5).
$$
From Equation (\ref{eq:yj}) we then determine the values for the other diagonals
$$y_1 = \frac{c_1y_0+c_{0,1}}{c_0} = \frac{1\cdot 3 +2 }{5} = 1,
$$
$$y_2 = \frac{c_2y_0+c_{0,2}}{c_0} = \frac{3\cdot 3 +26 }{5} = 7,
$$
and
$$y_3 = \frac{c_3y_0+c_{0,3}}{c_0} = \frac{1\cdot 3 +12 }{5} = 3.
$$
This leads to the frieze with coefficients given in Figure \ref{fig:fwcext2},
where for clarity we only include those labels which we just computed. Note that the 
original square $\mathcal{C}$ forms the bottom half of the hexagon.

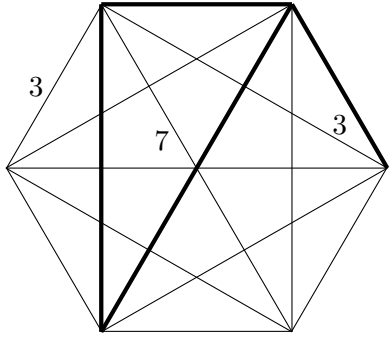
\begin{figure} 
\begin{tikzpicture}[auto]
    \node[name=s, draw, shape=regular polygon, regular polygon sides=6, minimum size=5cm] {};
    \draw[ultra thick] (s.corner 2) to (s.corner 4);
    \draw[ultra thick] (s.corner 1) to (s.corner 4);
    \draw[ultra thick] (s.corner 1) to (s.corner 2);
    \draw[ultra thick] (s.corner 1) to (s.corner 6);
    
    \draw[] (s.corner 1) to (s.corner 4);
    \draw[] (s.corner 2) to (s.corner 4);
    \draw[] (s.corner 2) to (s.corner 5);
    \draw[] (s.corner 3) to (s.corner 5);
    \draw[] (s.corner 1) to (s.corner 5);
     \draw[] (s.corner 3) to (s.corner 6);
    \draw[] (s.corner 4) to (s.corner 6);
     \draw[] (s.corner 2) to (s.corner 6);
    \draw[] (s.corner 1) to (s.corner 3);
    
    \draw[shift=(s.side 2)]  node[left]  {{$3$}};
    \draw[shift=(s.side 3)]  node[above right=48pt]  {{$7$}};
    \draw[shift=(s.side 5)]  node[above=40pt]  {{$3$}};
   \end{tikzpicture}
   \caption{Second step of the extension of $\mathcal{C}$.
   \label{fig:fwcext2}}
\end{figure}

\medskip

The third step in the extension procedure for the edge with label 12 in $\mathcal{C}$ 
is to extend the new boundary edge with label $3$. Hence we set $c_0=3$, $c_1=1$, $c_2=7$,
$c_3=3$ and $c_4=1$. For the only relevant prime number $p=3=c_0$
we have $m=0$ and we choose $i_3=4$. 

Lemma \ref{lem:yip} yields restrictions for $j=1$ and $j=2$, in both cases imposing
that $y_{i_3}\not\equiv 1 \,(\operatorname{mod}\,3)$. So we have to choose 
$y_{i_3}\equiv 2 \,(\operatorname{mod}\,3)$. Then from Lemma \ref{lem:yj} we obtain
$$y_0 \equiv (c'_{i_3})^{-1} \left( \frac{c_0}{1} y_{i_3} - c_{4,0}{1}\right)
\equiv 1\cdot (3\cdot 2-5) \equiv 1\,(\operatorname{mod}\,3).
$$
Equation \ref{eq:yj} gives the following values for the diagonals $y_j$, 
$$y_1=\frac{c_1y_0+c_{0,1}}{c_0} = \frac{1\cdot 1 +2}{3} = 1\hskip0.5cm
\mbox{and}\hskip0.5cm
y_2=\frac{c_2y_0+c_{0,2}}{c_0} = \frac{7\cdot 1 +26}{3} = 11,
$$
$$y_3=\frac{c_3y_0+c_{0,3}}{c_0} = \frac{3\cdot 1 +12}{3} = 5\hskip0.5cm
\mbox{and}\hskip0.5cm
y_4=\frac{c_4y_0+c_{0,4}}{c_0} = \frac{1\cdot 1 +5}{3} = 2.
$$
This leads to the frieze with coefficients given in Figure \ref{fig:fwcext3},
where again for clarity we only show a few of the diagonals and only the labels
we just computed and the remaining boundary labels not equal to 1.

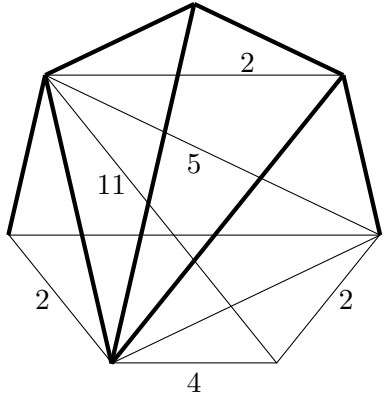
\begin{figure} 
\begin{tikzpicture}[auto]
    \node[name=s, draw, shape=regular polygon, regular polygon sides=7, minimum size=5cm] {};
    \draw[ultra thick] (s.corner 1) to (s.corner 2);
    \draw[ultra thick] (s.corner 2) to (s.corner 3);
    \draw[ultra thick] (s.corner 1) to (s.corner 7);
    \draw[ultra thick] (s.corner 1) to (s.corner 4);
    \draw[ultra thick] (s.corner 7) to (s.corner 4);
    \draw[ultra thick] (s.corner 7) to (s.corner 6);
    \draw[ultra thick] (s.corner 2) to (s.corner 4);
    
    \draw[] (s.corner 1) to (s.corner 4);
    \draw[] (s.corner 2) to (s.corner 4);
    \draw[] (s.corner 2) to (s.corner 5);
     \draw[] (s.corner 3) to (s.corner 6);
    \draw[] (s.corner 4) to (s.corner 6);
     \draw[] (s.corner 2) to (s.corner 6);
    \draw[] (s.corner 2) to (s.corner 7);

    \draw[shift=(s.side 3)]  node[left]  {{$2$}};
    \draw[shift=(s.side 4)]  node[below]  {{$4$}};
    \draw[shift=(s.side 5)]  node[right]  {{$2$}};
    \draw[shift=(s.corner 4)]  node[above=60pt]  {{$11$}};
    \draw[shift=(s.side 4)]  node[above=68pt]  {{$5$}};
    \draw[shift=(s.side 7)]  node[below left=2pt]  {{$2$}};
   \end{tikzpicture}
   \caption{Third step of the extension of $\mathcal{C}$.
   \label{fig:fwcext3}}
\end{figure}

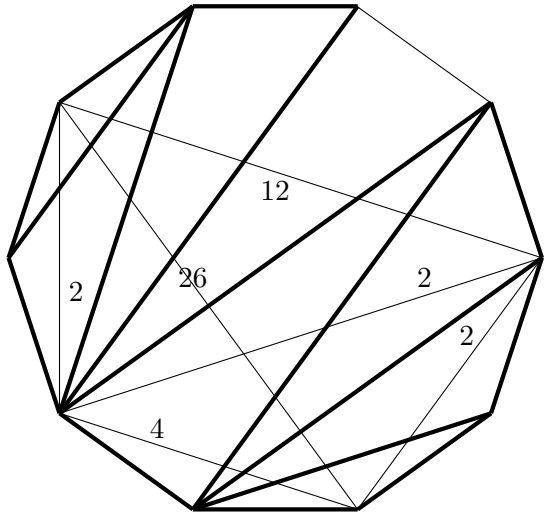
\begin{figure} 
\begin{tikzpicture}[auto]
    \node[name=s, draw, shape=regular polygon, regular polygon sides=10, minimum size=7cm] {};
    
    \draw[ultra thick] (s.corner 2) to (s.corner 4);
    \draw[ultra thick] (s.corner 2) to (s.corner 5);
    \draw[ultra thick] (s.corner 5) to (s.corner 1);
    \draw[ultra thick] (s.corner 5) to (s.corner 10);
    \draw[ultra thick] (s.corner 6) to (s.corner 10);
    \draw[ultra thick] (s.corner 6) to (s.corner 8);
    \draw[ultra thick] (s.corner 6) to (s.corner 9);
    \draw[ultra thick] (s.corner 1) to (s.corner 2);
    \draw[ultra thick] (s.corner 2) to (s.corner 3);
    \draw[ultra thick] (s.corner 3) to (s.corner 4);
    \draw[ultra thick] (s.corner 4) to (s.corner 5);
    \draw[ultra thick] (s.corner 5) to (s.corner 6);
    \draw[ultra thick] (s.corner 6) to (s.corner 7);
    \draw[ultra thick] (s.corner 7) to (s.corner 8);
    \draw[ultra thick] (s.corner 8) to (s.corner 9);
    \draw[ultra thick] (s.corner 9) to (s.corner 10);

    \draw[] (s.corner 5) to (s.corner 9);
    \draw[] (s.corner 3) to (s.corner 9);
    \draw[] (s.corner 3) to (s.corner 5);
    \draw[] (s.corner 5) to (s.corner 7);
    \draw[] (s.corner 3) to (s.corner 7);
    \draw[] (s.corner 7) to (s.corner 9);
    
    \draw[shift=(s.side 8)]  node[left=12pt]  {{$2$}};
    \draw[shift=(s.side 6)]  node[above=113pt]  {{$12$}};
    \draw[shift=(s.side 5)]  node[above right=7pt]  {{$4$}};
    \draw[shift=(s.side 4)]  node[above right=13pt]  {{$2$}};
    \draw[shift=(s.corner 6)]  node[above=80pt]  {{$26$}};
    \draw[shift=(s.side 10)]  node[below=77pt]  {{$2$}};
   \end{tikzpicture}
   \caption{The frieze with coefficients $\mathcal{C}$ as a subpolygon of a Conway-Coxeter frieze.
   \label{fig:cc}}
\end{figure}

Note that we have now completed the extension for the boundary edge with label 12
in the original frieze with coefficients $\mathcal{C}$.
It now remains to apply the same procedure to the other boundary edges with labels
$2$, $4$ and $2$. We leave the computations to the reader. Eventually, one can find
the triangulation of a decagon given in Figure \ref{fig:cc}, containing the original
frieze with coefficients $\mathcal{C}$ as a subpolygon.

\end{document}